\documentclass[12pt,reqno]{article}

\usepackage[usenames]{color}
\usepackage{amssymb}
\usepackage{amsmath}
\usepackage{amsthm}
\usepackage{amsfonts}
\usepackage{amscd}
\usepackage{graphicx}
\usepackage{todonotes}

\usepackage[colorlinks=true,
linkcolor=webgreen,
filecolor=webbrown,
citecolor=webgreen]{hyperref}

\definecolor{webgreen}{rgb}{0,.5,0}
\definecolor{webbrown}{rgb}{.6,0,0}

\usepackage{color}
\usepackage{fullpage}
\usepackage{float}

\usepackage{graphics}
\usepackage{latexsym}
\usepackage{epsf}
\usepackage{breakurl}

\newcommand{\seqnum}[1]{\href{https://oeis.org/#1}{\rm \underline{#1}}}

\def\Zee{\mathbb{Z}}
\def\suchthat{\, : \,}

\makeatletter
\def\@fnsymbol#1{\ensuremath{\ifcase#1\or *\or \mathsection\or \mathparagraph\or \|\or **\or \ddagger\ddagger \or \dagger\dagger \else\@ctrerr\fi}}
\makeatother

\title{New properties of the $\varphi$-representation of integers}

\author{Jeffrey Shallit\footnote{Research supported by NSERC grant 2024-03725.} \ and Ingrid Vukusic$^*$\footnote{Author's current address:  Department of Mathematics,
University of York,
York, North Yorkshire YO10 5GH,
United Kingdom.  E-mail
\href{mailto:ingrid.vukusic@york.ac.uk}{\tt ingrid.vukusic@york.ac.uk}.}\\
School of Computer Science\\
University of Waterloo\\
Waterloo, ON N2L 3G1 \\
Canada\\
\href{mailto:shallit@uwaterloo.ca}{\tt shallit@uwaterloo.ca}}

\begin{document}

\maketitle

\theoremstyle{plain}
\newtheorem{theorem}{Theorem}
\newtheorem{corollary}[theorem]{Corollary}
\newtheorem{lemma}[theorem]{Lemma}
\newtheorem{proposition}[theorem]{Proposition}

\theoremstyle{definition}
\newtheorem{definition}[theorem]{Definition}
\newtheorem{example}[theorem]{Example}
\newtheorem{conjecture}[theorem]{Conjecture}

\theoremstyle{remark}
\newtheorem{remark}[theorem]{Remark}

\begin{abstract}
We prove a few new properties of the $\varphi$-representation of integers, where $\varphi = (1+\sqrt{5})/2$.  In particular, we prove a 2012 conjecture of Kimberling.  As software assistants, we used the {\tt Walnut} theorem-prover, and in one proof, {\tt ChatGPT 5}.
\end{abstract}

\section{Introduction}

Let $\varphi = (1+\sqrt{5})/2$ be the golden ratio.  Bergman \cite{Bergman:1957} showed that every positive integer $x$ can be represented as a finite sum of distinct (possibly negative) powers of $\varphi$; we write
$x = \sum_{a \leq i \leq b} e_i \varphi^i$ for $a, b \in \Zee$ with
$a \leq b$ and $e_i \in \{0,1\}$.  Furthermore, this representation is unique, provided that $e_i e_{i-1} = 0$ for $a+1 \leq i\leq b$.  Also see \cite{Rousseau:1995}.

More generally, R\'enyi \cite{Renyi:1957} proved an analogous result for representing arbitrary positive real numbers.   Here the representation could be infinite, and to get uniqueness we have to impose the further rule
that there is no index $j$ such that
    $e_{j-2i} = 1$, $e_{j-2i-1} = 0$ for all $i \geq 0$.

\begin{example}
Here are a few examples of $\varphi$-representations.
\begin{align*}
    5 &= \varphi^3 + \varphi^{-1} + \varphi^{-4} \\
    1/5 &= \sum_{i \geq 0} (\varphi^{-20i-4} + \varphi^{-20i-7} +
    \varphi^{-20i-9} +
    \varphi^{-20i-11} +
    \varphi^{-20i-14} +
    \varphi^{-20i-17} ) \\
    \sqrt{5} &= \varphi^1 + \varphi^{-1} \\
    \pi &= \varphi^2 + \varphi^{-2} +
    \varphi^{-5} + \varphi^{-7} + \varphi^{-9} + \varphi^{-12} +
    \varphi^{-16} + 
    \varphi^{-18} +
    \varphi^{-20} +
    \varphi^{-22} + 
    \varphi^{-28} + \cdots 
\end{align*}
\end{example}
For more papers discussing this representation of real numbers, see
\cite{Rousseau:1995,Frougny&Sakarovitch:1999,Sanchis&Sanchis:2001,Dekking:2020,Dekking:2020b,Dekking:2021,Dekking&vanLoon:2023,Dekking:2024,Dekking&vanLoon:2024}.  In fact, $\varphi$-representation is a special case of a much more general numeration system, called $\beta$-numeration.  See, for example, \cite{Renyi:1957,Parry:1960,Bertrand-Mathis:1989,Frougny:1992d,Frougny&Solomyak:1992}.

If a $\varphi$-representation is finite, we can write it as a binary string with a decimal point, in the form
$e_b \cdots e_1 e_0 . e_{-1} e_{-2}  \cdots e_{a}$
representing $\sum_{a \leq i \leq b} e_i \varphi^i$.
If $e_b \not=0$ and $e_a \not=0$, we call this
representation {\it canonical}.  We call $e_b \cdots e_1 e_0$ the {\it positive part}
and $e_{-1} e_{-2} \cdots e_a$ the {\it negative part}, referring to the exponents that appear.
Thus, for example,
the canonical representation of $5$ is the string $1000.1001$.  

Finite $\varphi$-representations correspond
exactly to the non-negative members of $\Zee[\varphi]$.  This was mentioned in Bergman \cite{Bergman:1957} and generalized in \cite{Frougny&Solomyak:1992}. 
In this note we establish some novel properties of the $\varphi$-representations of natural numbers.  We use both ``conventional'' arguments and the {\tt Walnut} theorem-prover.

\section{Shevelev's set and Kimberling's conjecture}

In 2010, Vladimir Shevelev introduced an interesting set $S$ of natural numbers 
$$\{ 1,3,4,7,8,10,11,18,19,21,22,25,26,28,29,47,\ldots\}$$ based on $\varphi$-representation, as  sequence \seqnum{A178482} in the {\it On-Line Encyclopedia of Integer Sequences}
(OEIS) \cite{oeis}.   This set consists of those integers whose $\varphi$-representation is {\it anti-palindromic}; that is, $t$ appears as an exponent in the
$\varphi$-representation if and only if $-t$ also appears.  As an example, 
$25 = \varphi^6 + \varphi^4 + \varphi^{-4} + \varphi^{-6}$ belongs to $S$.

In the entry for \seqnum{A178482}, Clark Kimberling conjectured in 2012 that $S$ consists precisely of those integers $n$
such that if we take the 
$\varphi$-representation of $n$
and replace each exponent $t$ appearing in it by $2t$, then
the result is an integer as well.  For
example, $10 \in S$ and
its $\varphi$-representation
is $\varphi^4 + \varphi^2 + \varphi^{-2} + \varphi^{-4}$.  When
we replace each exponent $t$ by $2t$,
we get $\varphi^8 + \varphi^4 + \varphi^{-4} + \varphi^{-8} = 54$.
On the other hand, $9 \not\in S$.
Its $\varphi$-representation is
$\varphi^4 + \varphi^1 + \varphi^{-2} + \varphi^{-4}$; when
we double the exponents, we get
$\varphi^8 + \varphi^2 + \varphi^{-4} + \varphi^{-8} = 52 - \sqrt{5}$, which is not an integer.

In this section we prove Kimberling's conjecture.
First, we need two lemmas.
\begin{lemma}
Suppose the $\varphi$-expansion of a non-negative real number $x$ is antipalindromic.  Then $x$ is an integer if and only if all the exponents appearing in its expansion are even.
\label{one}
\end{lemma}

\begin{proof}
Define $\overline{\varphi} = -\varphi^{-1} = (1-\sqrt{5})/2$.   Recall that the Fibonacci numbers $F_i$ are defined
by $(\varphi^i - \overline{\varphi}^i)/\sqrt{5}$ and
the Lucas numbers $L_i$ are defined
by $\varphi^i + \overline{\varphi}^i$.
Both of these sequences consist of integers.  

If the $\varphi$-representation of $x$ is antipalindromic, we can write write the
set $\{ i \suchthat e_i(n) = 1 \}$
as the disjoint union of four sets
$${\cal E}, {\cal O}, 
-{\cal E}, -{\cal O}$$
and possibly also $\{0\}$,
where $\cal E$ is a set of even positive integers and $\cal O$ is a set of odd positive integers.
Notice that if $i$ is even, then
$\varphi^i + \varphi^{-i} = L_i$.
On the other hand, if $i$ is odd,
then $\varphi^i + \varphi^{-i} =
\varphi^i - \overline{\varphi}^i =
F_i \sqrt{5}$.
Thus we can express $x$ as the sum of
certain Lucas numbers (corresponding to ${\cal E}$ and $-{\cal E}$), plus $\sqrt{5}$ times the sum of certain Fibonacci numbers (corresponding to
${\cal O}$ and $-{\cal O}$), and possibly $1$ (if $0$ is present as an exponent).   Thus $x$ is an integer if and only if ${\cal O}$ is empty, which corresponds to all exponents in $x$'s $\varphi$-representation being even.
\end{proof}

\begin{lemma}
Suppose the $\varphi$-expansion of a non-negative real number $x$ has all even exponents.   Then $x$ is an integer if and only if its $\varphi$-expansion is antipalindromic.
\label{two}
\end{lemma}

\begin{proof}
First, suppose $x$ is an integer whose $\varphi$-expansion consists of all even exponents.   Let ${\cal E}_{\text{pos}}$ be the set of positive exponents and
${\cal E}_{\text{neg}}$ be the set of negative exponents.  Consider the number
$n$ whose $\varphi$-representation
consists of powers of $\varphi$ in ${\cal E}_{\text{pos}}$ and
$-{\cal E}_{\text{pos}}$, and possibly $\{0\}$ if $0$ is present as an exponent in the expansion of $x$.  By Lemma~\ref{one} we see that $n$ is an integer.   Furthermore
$0 \leq |n - x| < \varphi^{-2} + \varphi^{-4} + \varphi^{-6} + \cdots = \varphi^{-1} < 1$.
Since $n$ is an integer, we must
have $n = x$ and its expansion
consists of the exponents ${\cal E}_{\text{pos}}$,
$-{\cal E}_{\text{pos}}$, and possibly $0$.  So it is antipalindromic.

The converse implication follows directly from Lemma~\ref{one}.
\end{proof}

\begin{remark}
By combining Lemmas~\ref{one} and \ref{two}, we get that the $\varphi$-expansion of an integer $n$ is antipalindromic if and only if all exponents are even.
\end{remark}

We can now prove Kimberling's conjecture.
\begin{theorem}
The set $S$ consists of those natural numbers
$n$ such that
$\sum_{a \leq i \leq b} e_i(n) \varphi^{2i}$
is an integer,
where $n = \sum_{a \leq i \leq b} e_i(n) \varphi^i$ is the 
$\varphi$-representation of $n$.
\end{theorem}

\begin{proof}
Let $\sum_{a \leq i \leq b} e_i \varphi^i$ be the $\varphi$-representation of an integer $n$, and define
$x = \sum_{a \leq i \leq b} e_i \varphi^{2i}$.

Suppose $n \in S$.  Then its $\varphi$-representation is antipalindromic. Then $x$ is also 
antipalindromic, and all exponents
are even.  So by Lemma~\ref{one}, the real number $x$ is an integer.

Now suppose $x$ is an integer.  By the definition of $x$, all the exponents appearing in its $\varphi$-representation are even.  So by Lemma~\ref{two}, we see that $x$ is antipalindromic.  Hence so is $n$, and thus $n \in S$.
\end{proof}

\section{Other results}

In this section we prove a few other novel results about $\varphi$-representations of the natural numbers.
One tool that we use is {\tt Walnut}, a free, open-source prover for first-order statements about automatic
    sequences \cite{Mousavi:2021,Shallit:2023}.   In the paper 
    \cite{Shallit:2025}, {\tt Walnut} code is developed for manipulating the $\varphi$-expansions of natural numbers.

To explain what {\tt Walnut} can do and how it does it, we need to talk about the Zeckendorf (or Fibonacci) representation of natural numbers.  In this system \cite{Lekkerkerker:1952,Zeckendorf:1972} we write a natural number $n$ as a sum of distinct Fibonacci numbers $F_i$ for $i \geq 2$.  This representation is unique, provided we never use two adjacent Fibonacci numbers in it.  We abbreviate such a representation as a binary string
$b_t b_{t-1} \cdots b_2$ for the integer
$[b_t b_{t-1} \cdots b_2]_F := \sum_{2 \leq i \leq t} b_i F_i$; it is called {\it canonical} if $b_t= 1$. Thus, for example, $11 = F_6 + F_4$ is represented by the binary string
$10100$.   

In the paper \cite{Shallit:2025}, the first author showed how to construct a $39$-state finite automaton ${\tt saka}$ that takes three same-length binary strings in parallel as input, say $w,y,z$,
and accepts if $[w]_F$ has $\varphi$
representation $y.z^R$, where $z^R$ denotes the reversal of the string $z$. (The automaton {\tt saka} can be downloaded from \url{https://cs.uwaterloo.ca/~shallit/Papers/allcode.tar}.) In order to force the strings to have the same length, we pad the shorter of $w$, $y$, and $z$ with leading zeroes, if needed.
Furthermore, any number of leading zeroes is allowed.
Another way to think of this representation $y.z^R$ is
as a concatenation of the blocks $[0,0],[0,1],[1,0],[1,1]$, where the first components spell out $y$
and the second components spell out $z$.  We call this the {\it folded representation}.

Although the automaton {\tt saka} takes binary strings as input, in {\tt Walnut} we can call the automaton with numerical inputs, and sometimes this is quite useful.  
For example, if we write
\begin{verbatim}
eval check "?msd_fib $saka(9,10,7)":
\end{verbatim}
in {\tt Walnut}, this is asking whether $10010.0101$
is the $\varphi$-representation of $9$ (since
$[10]_F = 10010$ and $[7]_F = 1010$). Here {\tt Walnut} answers {\tt TRUE}.

We briefly recap some of the features of {\tt Walnut}.  For more information about it, see \cite{Shallit:2023}.  
\begin{itemize}
    \item The notation {\tt msd\_fib} says that all integers in the expression are represented in Zeckendorf representation.
    \item We can add and subtract integers in {\tt Walnut}, and multiply and divide by integer constants, but we cannot multiply two variables.
    \item The {\tt reg} command allows one to define an automaton using a regular expression.
    \item The symbol {\tt \&} means logical {\tt AND}, {\tt |} means logical {\tt OR}, {\tt =>} is logical implication, {\tt <=>} is logical {\tt IFF},
    {\tt \char'176} is logical {\tt NOT}.
    \item The letter {\tt A} is the universal quantifier $\forall$ and the letter {\tt E} is the existential quantifier $\exists$.
    \end{itemize}

\begin{example}
The following {\tt Walnut} code generates an automaton accepting those pairs $(x,y)$ such that $x.y^R$ is the $\varphi$-representation of some integer $n \geq 0$.  The automaton is depicted in Figure~\ref{fig1}.
\begin{verbatim}
def sk "?msd_fib En $saka(n,x,y)":
\end{verbatim}
\begin{figure}[htb]
\begin{center}
    \includegraphics[width=6.75in]{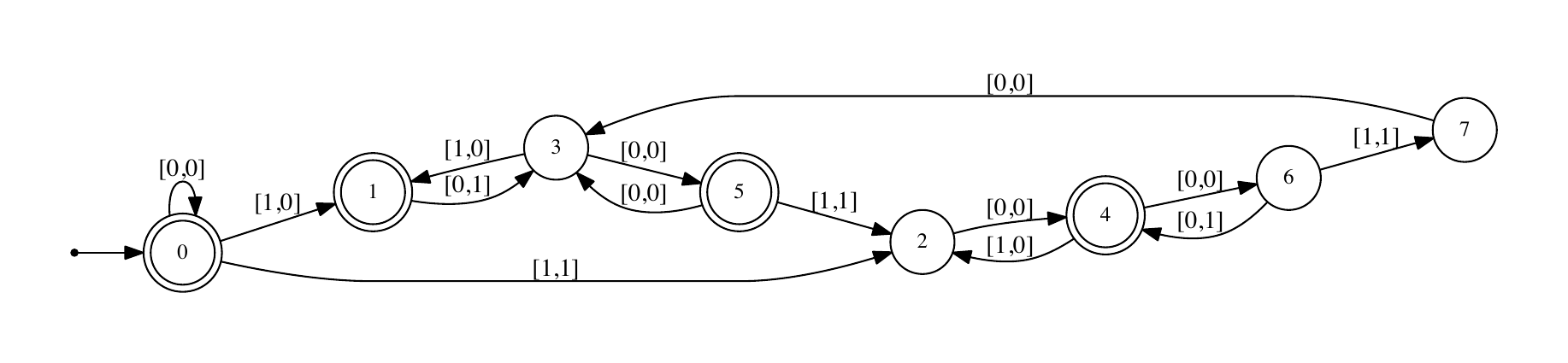}
\end{center}
\caption{Automaton accepting $(x,y)$ such that $x.y^R$ is the $\varphi$-representation of some integer $n \geq 0$.}
\label{fig1}
\end{figure}
\end{example}

As a warmup let us prove the following result.
\begin{proposition}
Let $n$ be a positive integer with canonical
representation $\sum_{a \leq i \leq b} e_i \varphi^i$.
If $b$ is odd then $a = -(b+1)$, and if $b$ is even then $a = -b$.
\label{simple}
\end{proposition}

\begin{proof}
In terms of the folded representation this means that for $n \geq 1$,
an even-length representation must begin with $[1,1]$,
and for $n \geq 2$ an odd-length representation must begin with
$[1,0][0,1]$.  We can verify this with {\tt Walnut} as follows:
\begin{verbatim}
reg start11 msd_fib msd_fib "[0,0]*[1,1]([0,0]|[0,1]|[1,0]|[1,1])*":
reg start1001 msd_fib msd_fib "[0,0]*[1,0][0,1]([0,0]|[0,1]|[1,0]|[1,1])*":
reg evenl msd_fib msd_fib "[0,0]*([0,1]|[1,0]|[1,1])
(([0,0]|[0,1]|[1,0]|[1,1])([0,0]|[0,1]|[1,0]|[1,1]))*([0,0]|[0,1]|[1,0]|[1,1])":
reg oddl msd_fib msd_fib "[0,0]*([0,1]|[1,0]|[1,1])
(([0,0]|[0,1]|[1,0]|[1,1])([0,0]|[0,1]|[1,0]|[1,1]))*":
eval check_num_even "?msd_fib An,x,y (n>=1 & $saka(n,x,y) & $evenl(x,y)) =>
$start11(x,y)":
eval check_num_odd "?msd_fib An,x,y (n>=2 & $saka(n,x,y) & $oddl(x,y)) => 
$start1001(x,y)":
\end{verbatim}
And {\tt Walnut} returns {\tt TRUE} for both.
\end{proof}

We have seen that the set $S$ consists of those integers $n$ whose
$\varphi$-representation consists of all even powers.
We can find an automaton accepting the set $S$ as follows:
\begin{verbatim}
reg shiftr {0,1} {0,1} "([0,0]|[1,0][1,1]*[0,1])*(()|[1,0][1,1]*)":
# accepts x,y such that y is the right shift of the bits of x
def a178482 "?msd_fib Ex,y $saka(n,x,y) & $shiftr(x,y)":
\end{verbatim}
It is depicted in Figure~\ref{fig2}.
\begin{figure}[htb]
\begin{center}
    \includegraphics[width=6.75in]{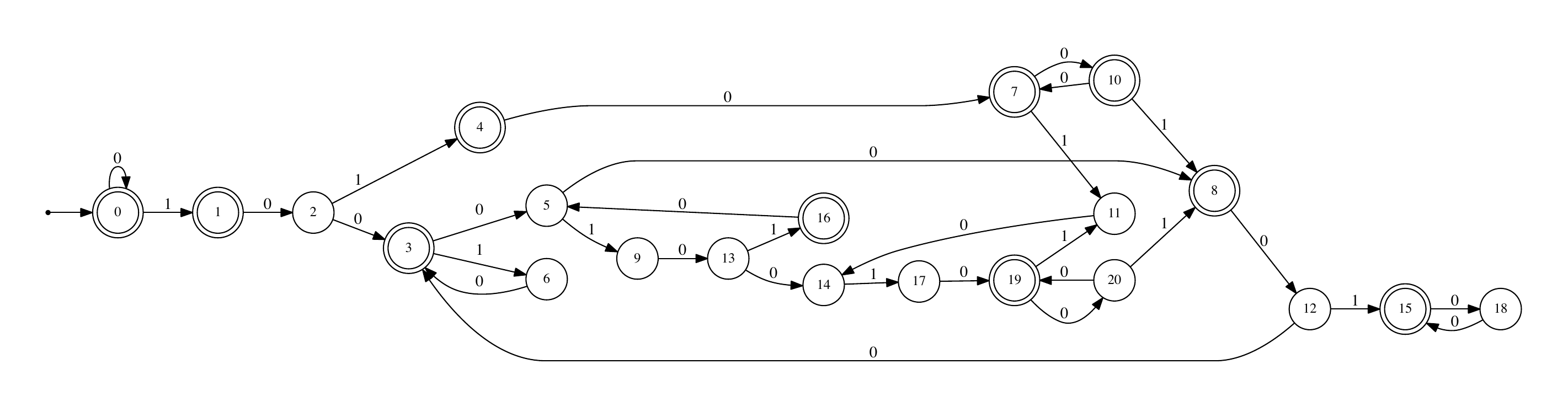}
\end{center}
\caption{Automaton accepting the Zeckendorf representations of the members of sequence \seqnum{A178482}.}
\label{fig2}
\end{figure}

This raises the question of whether there are any integers whose $\varphi$-representation consists of all odd powers.  This is answered (negatively) in the following version of a theorem of Dekking \cite[Thm.~9]{Dekking:2024}.  It says that the negative parts of canonical $\varphi$-representations consist of all binary strings without two consecutive $1$'s that end with a $1$ and whose length is even.
\bigbreak
\begin{theorem}\label{thm:last_index}
\leavevmode
\begin{itemize}
    \item[(a)] 
Suppose $n \geq 2$ is a positive integer with canonical 
$\varphi$-representation $x.y$ for strings $x,y$.
Then $|y|$ is positive and even.

    \item[(b)]  Let $y$ be the canonical Zeckendorf representation of some positive integer such that $|y|$ is even.  Then there exists an integer
    $n$ and a string $x$ such that the $\varphi$-representation of $n$ is $n = x.y^R$.
    
\item[(c)] 
If $n \geq 2$ and $L_{2i-1} < n \leq L_{2i+1}$, then the smallest exponent in the $\varphi$-representation of $n$ 
is $-2i$.
\end{itemize}
\end{theorem}

\begin{proof}
    We use {\tt Walnut}.  For claim (a), we use the following code:
\begin{verbatim}
reg largesteven msd_fib "0*1(0|1)((0|1)(0|1))*":
eval thm_a "?msd_fib An,x,y (n>=2 & $saka(n,x,y)) => $largesteven(y)":
\end{verbatim}
And {\tt Walnut} returns {\tt TRUE}.

For claim (b), we write
\begin{verbatim}
eval thm_b "?msd_fib Ay Ez ($largesteven(y) & y=z) => En,x $saka(n,x,z)":
\end{verbatim}
and once again we get {\tt TRUE}.

For claim (c), we recall the formula $L_n = F_{n-1} + F_{n+1}$ for $n\geq 1$, as well as the equalities $L_0 = F_3$ and $L_2 = F_4$. We use the
following code:
\begin{verbatim}
reg isevenlucas msd_fib "0*10|0*100|0*1010(00)*":
# is x an even-indexed Lucas number?
reg shiftl {0,1} {0,1} "([0,0]|[0,1][1,1]*[1,0])*":
# accepts x,y such that y is the left shift of the bits of x
reg shiftr {0,1} {0,1} "([0,0]|[1,0][1,1]*[0,1])*(()|[1,0][1,1]*)":
# accepts x,y such that y is the right shift of the bits of x
reg fibluc msd_fib msd_fib "[0,0]*([0,1][0,0][1,0])|([0,1][1,0][0,1][0,0]*)":
# accepts (x, y) if x = F_i and y = L_i for i>=2
reg largest_dig msd_fib msd_fib "[0,0]*[1,1]([1,0]|[0,0])*":
# matchest largest digit of input
eval thm_c "?msd_fib Ar,s,t,u,n,x,y (($isevenlucas(r) & 
   $shiftr(r,s) & $shiftl(r,t) & $fibluc(u,t) & r>=3 
   & $saka(n,x,y) & n>=2 & s<n & n<=t) => $largest_dig(y,u))":
\end{verbatim}
Here is the meaning of the variables:
\begin{itemize}
    \item $r = L_{2i}$ for some $i \geq 1$;
    \item $s = L_{2i-1}$;
    \item $t = L_{2i+1}$;
    \item $u = F_{2i+1}$.
\end{itemize}
The last expression returns {\tt TRUE}; it checks that if $L_{2i-1} < n \leq L_{2i+1}$, then
$\varphi^{-2i}$ is the least power of $\varphi$ appearing in the
$\varphi$-representation $x.y^R$ of $n$.  This means that the leftmost
$1$ in the string $y^R$ corresponds to $F_{2i+1}$, if we considered
$y^R$ as a Zeckendorf expansion.  So the largest $1$ digit in $y^R$, when viewed as a Zeckendorf representation of $n$, corresponds to $F_{2i+1}$, which is the meaning of the implicand.
\end{proof}

We also provide a more standard, non-{\tt Walnut} proof. The proof of the first claim was shown to us by ChatGPT 5.

\begin{proof}
(a) Suppose $n \geq 1$ is an integer where the least exponent $a$ appearing in the
$\varphi$-represen\-tation is odd.
Then $a+1$ does not appear.
Apply Galois conjugation, replacing
$\varphi$ in the representation by
$\overline{\varphi}$, obtaining that $n$ is also a sum of the same powers of $\overline{\varphi}$.
Now since $a$ is odd, we have
$\overline{\varphi}^a < 0$.  Thus
the next larger positive term one could have in this new representation is $\overline{\varphi}^{a+3}$.  It 
follows that $n \leq \overline{\varphi}^a + \overline{\varphi}^{a+3} + \overline{\varphi}^{a+5} + \cdots = \overline{\varphi}^a (1-\varphi^{-2}) < 0$, a contradiction.

\medskip

(b) By induction on the (even) length of the string $y$. First, note that for the empty string $\varepsilon$, we can take $n=1=1.\varepsilon^R$. Next, for $y = 10$, we can take $n=2=10.01$. Now assume that $|y|=c\geq 4$ is even, and that $y$ starts with a $1$. 
We distinguish between three cases according to the number of $0$'s following the leading $1$ in $y$.

\textit{Case 1}: The leading $1$ in $y$ is followed by at least three $0$'s. We replace the prefix $1000$ in $y$ by $10$ to obtain a new string $y'$ with $|y'| = c-2$. 
Then by induction there is some $n'$ whose $\varphi$-representation is $x'.{y'}^R$.
Now consider 
$n = n' - (\varphi^{c-2} + \varphi^{-(c-2)}) + (\varphi^{c} + \varphi^{-c}) = n' - L_{c-2} + L_c$, which is clearly a positive integer.
By construction (note that $- \varphi^{c-2} + \varphi^{c} = \varphi^{c-1}$, which only impacts the positive part of the representation) we know that the $\varphi$-representation of $n$ has negative part $y$.

\textit{Case 2}: The leading $1$ in $y$ is followed by exactly two $0$'s. Since $|y|=c\geq 4$, we know that $y$ has the prefix $1001$. We replace it by $10$ to obtain a new string $y'$ with $|y'| = c-2$. Again, by induction there is some $n'$ whose $\varphi$-representation is $x'.{y'}^R$. 
Note that $\varphi^{-c} +  \varphi^{-(c-3)} - \varphi^{-(c-2)} = \varphi^{-(c-2)}$.
Therefore, we can consider
$n = n' + \varphi^{-(c-2)} + \varphi^{c-2} = n' + L_{c-2}$.
This is again an integer, and by construction the $\varphi$-representation of $n$ has negative part $y$.

\textit{Case 3}: The leading $1$ in $y$ is followed by exactly one $0$, i.e., $y$ has the prefix $101$. Then we can replace $101$ by $1$ to obtain $y'$, and finish as in the previous cases, simply setting $n = n' + L_c$.

\medskip

(c) Let $n \geq 2$ be an integer. By the first part we know that the smallest exponent in its $\varphi$-representation is even, say $-2i$. We need to show that $L_{2i-1} < n \leq L_{2i+1}$. 
By Galois conjugation we have $n = \varphi^{2i} \pm \dots$, where the remaining terms have exponents $\leq 2i - 2$ and all even powers occur with positive sign and all odd powers occur with negative sign.
Thus we get on the one hand
\[
    n
    \leq \varphi^{2i} + \varphi^{2i-2} + \varphi^{2i-4} + \dots
    = \varphi^{2i+1}
    = L_{2i+1} + \varphi^{-(2i+1)},
\]
which implies
\[
    n \leq L_{2i+1}.
\]
On the other hand,
\[
    n
    \geq \varphi^{2i} - \varphi^{2i-3} - \varphi^{2i-5} - \dots
    = \varphi^{2i} - \varphi^{2i-2}
    = \varphi^{2i-1}
    > L_{2i-1},
\]
as desired.
\end{proof}

So while there are expansions of integers with all exponents even, there are none with all exponents odd.  However, there are infinitely many $n$ whose expansions have all odd exponents, except for one that is even (the smallest).  The first few terms of this set are
$$ 1, 2, 5, 6, 12, 13, 16, 17, 30, 31, 34, 35, 41, 42, 45, 46, 77, \ldots .$$

\begin{theorem}
The Zeckendorf representations of those $n$ for which their $\varphi$-expansions have exactly one even exponent are accepted by the DFA
in Figure~\ref{fig7}.  Furthermore,
$n$ is in this set if and only if
$n-1$ is the (possibly empty) sum of distinct odd-indexed Lucas numbers $L_{2i+1}$ with $i \geq 0$.
\begin{figure}[H]
\begin{center}
\includegraphics[width=6.5in]{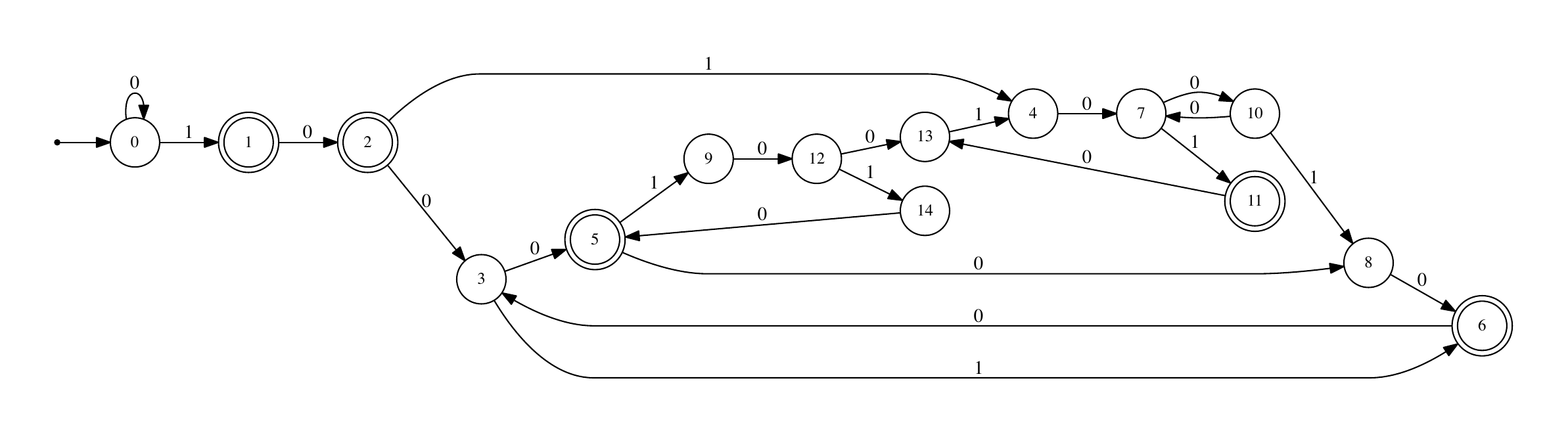}
\end{center}
\caption{Automaton accepting Zeckendorf representation of those $n$ whose $\varphi$-representation has exactly one even exponent.}
\label{fig7}
\end{figure}
\label{them7}
\end{theorem}

\begin{proof}
We use {\tt Walnut}. 
To create the automaton for the Zeckendorf representations of suitable $n$, recall that ${\tt saka}$ checks whether $n = x.y^R$, where the last position of the binary string $x$ corresponds to the exponent $0$, and the last position of the binary string $y$ corresponds to the exponent $-1$. 
Thus, $n$ having exactly one even exponent in its $\varphi$-representation corresponds to one of the following two cases:
\begin{itemize}
	\item $n$ has one non-negative even exponent and no negative even exponents (which only occurs for $n = 1$). If we say that the last position of $x$ is ``position $1$'', then this corresponds to $x$ having exactly one $1$ at an odd position (and even positions don't matter) and $y$ having no $1$'s at even positions (and odd positions don't matter).
	\item $n$ has no non-negative even exponents and exactly one negative even exponent. This corresponds to $x$ not having any $1$'s at odd positions and $y$ having exactly one $1$ at an even position.
\end{itemize}

\begin{verbatim}
reg noodd1 {0,1} "(()|0)((0|1)0)*":
# no 1 in an odd position from the end (end is position 1); 
# even positions don't count
reg noeven1 {0,1} "()|((()|0)((0|1)0)*(0|1))":
# no 1 in an even position from the end (end is position 1); 
# odd positions don't count
reg oneodd1 {0,1} "(()|0|1)(0(0|1))*1((0|1)0)*":
# has exactly one 1 in an odd position from the end (end is position 1); 
# even positions don't count
reg oneeven1 {0,1} "(()|0|1)(0(0|1))*1((0|1)0)*(0|1)":
# has exactly one 1 in an even position from the end (end is position 1); 
# odd positions don't count
def one_even_exponent "?msd_fib ($oneodd1(x) & $noeven1(y)) | 
   ($noodd1(x) & $oneeven1(y))":
def thm9 "?msd_fib Ex,y $saka(n,x,y) & $one_even_exponent(x,y)":
\end{verbatim}
This produces the automaton in
Figure~\ref{fig7}.  

For the second claim, we need an automaton that converts a representation as a sum of Lucas numbers into its equivalent Zeckendorf representation.  To do this, for a binary string
$b_t \cdots b_0$ we define
$[b_t \cdots b_0]_L = \sum_{0 \leq i \leq t} b_i L_i$.
The automaton ${\tt luctofib}$ will accept $(x,y)$ if $[x]_L = [y]_F$ (where, as usual, $x$ and $y$ are numerical values interpreted as the binary strings corresponding to their Zeckendorf representation). 
Note that $L_i = F_{i-1} + F_{i+1}$ and 
that the last digit in the Lucas number representation corresponds to $L_0$ whereas the last digit in the Fibonacci representation corresponds to $F_2$.
Therefore, a $1$ in the Lucas number representation corresponds to a $1$ shifted to the right by one position and another $1$ shifted to the right by three positions. 
Moreover, we need to be careful with the last three positions: $L_0 = 2, L_1 = 1 = F_2, L_2 = 3 = F_1 + 2$. 
Therefore, we need to check if the last or the third bit from the end in $y$ is $1$, and adjust accordingly.

\begin{verbatim}
reg end1 {0,1} "(0|1)*1":
# has a 1 at the last position
def hasbit3 "?msd_fib Et,u $shiftr(x,t) & $shiftr(t,u) & $end1(u)":
# has a 1 in 3rd bit from end
def bit1 "?msd_fib (y=2 & $end1(x))|(y=0 & ~$end1(x))":
# accepts (x,y) if y will "correct" the value of x for position 1
def bit3 "?msd_fib (y=1 & $hasbit3(x))|(y=0 & ~$hasbit3(x))":
# accepts (x,y) if y will "correct" the value of x for position 3
def luctofib "?msd_fib Et,u,v,w,z $shiftr(x,t) & $shiftr(t,u) & 
   $shiftr(u,v) & $bit1(x,w) & $bit3(x,z) & y=t+v+w+z":
def sum_odd_lucas "?msd_fib Ey $noodd1(y) & $luctofib(y,n)":
eval thm9b "?msd_fib An $sum_odd_lucas(n-1) <=> $thm9(n)":
\end{verbatim}
And {\tt Walnut} returns {\tt TRUE}.
\end{proof}

We also provide a more ``standard'' proof of the second claim.

\begin{proof}
The statement of the lemma is clear for $n = 1$, so let us only consider integers larger than $1$.

First assume that $n-1$ is the sum of distinct odd-indexed Lucas numbers
\[
    n - 1 = \sum_{i \in I} L_{2i + 1},
\]
where $I$ is some finite set of non-negative integers.
Then since $L_k = \varphi^k - \varphi^{-k}$ for odd $k$, we have 
\begin{equation}\label{eq:n_oddLucasSum}
    n = \left(\sum_{i \in I} \varphi^{2i + 1} \right) 
    + 1 - \sum_{i \in I} \varphi^{-(2i+1)}.
\end{equation}
Note that since $\varphi^{k} = \varphi^{k-1} + \varphi^{k-2}$ for all $k$, we have for $i \geq k$
\[
    \varphi^{2k} - \varphi^{-(2i+1)}
    = \varphi^{2k-1} + \varphi^{2k-3} + \dots +\varphi^{-(2i - 1)} + \varphi^{- (2i+2)}.
\]
We can apply this formula repeatedly (starting with $k=0$) in \eqref{eq:n_oddLucasSum} to see that the $\varphi$-representation of $n$ indeed only contains odd powers except for the smallest one.

Now assume conversely that the integer $x>1$ has exactly one even exponent in its $\varphi$-representation. 
By Theorem~\ref{thm:last_index} it must be the smallest exponent, and so in particular all positive exponents are odd, and $\varphi^0$ does not appear.
Let ${\cal O}_{\text{pos}}$ be the set of (odd) positive exponents in the $\varphi$-representation of $x$, and ${\cal A}_{\text{neg}}$ be the set of negative exponents. Consider the number
\[
    n = 1 + \sum_{k \in {\cal O}_{\text{pos}}} L_k.
\]
Then using $L_k = \varphi^k - \varphi^{-k}$ for odd $k$, we get
\begin{align*}
    x - n
    &= \sum_{k \in {\cal O}_{\text{pos}}} \varphi^{k} 
        + \sum_{k \in {\cal A}_{\text{neg}}} \varphi^{k}
    - \left( 1 + \sum_{k \in {\cal O}_{\text{pos}}} \varphi^k
    - \sum_{k \in -{\cal O}_{\text{pos}}} \varphi^k
    \right)\\
    & = -1 + \sum_{k \in {\cal A}_{\text{neg}}} \varphi^{k}
    + \sum_{k \in - {\cal O}_{ \text{pos}}} \varphi^{k}.
\end{align*}
By the properties of $\varphi$-representations, each of the two sums lies strictly between $0$ and $1$, and therefore $|n-x| < 1$. 
Thus, $x$ is an integer if and only if $x = n$, which is of the desired shape. 
\end{proof}

We might also examine those $n$
having exactly one odd exponent in their $\varphi$-expansion. The first
few elements of this set are
$$ 2, 9, 20, 27, 49, 56, 67, 74, 125, 132, 143, 150, 172, 179, 190, 197, \ldots .$$
For example,
\begin{align*}
    2 &= \varphi^1 + \varphi^{-2} \\
    9 &= \varphi^4 + \varphi^1 + \varphi^{-2} + \varphi^{-4} \\
    20 &= \varphi^6 + \varphi^1 +
    \varphi^{-2} + \varphi^{-6} .
\end{align*}
Examining these expansions suggest the following result.
\begin{theorem}
The $\varphi$-expansion of a natural number $n$ contains exactly one odd exponent if and only if $n$ is accepted
by the Fibonacci DFAO in
Figure~\ref{fig8}. Furthermore
a number $n$ has this property
if and only if $n-2$ is the (possibly empty) sum of
distinct even-indexed Lucas numbers
$L_{2i}$ with $i \geq 2$.  Finally, for these $n$, the odd exponent is equal to $1$.
\begin{figure}[H]
\begin{center}
\includegraphics[width=6.5in]{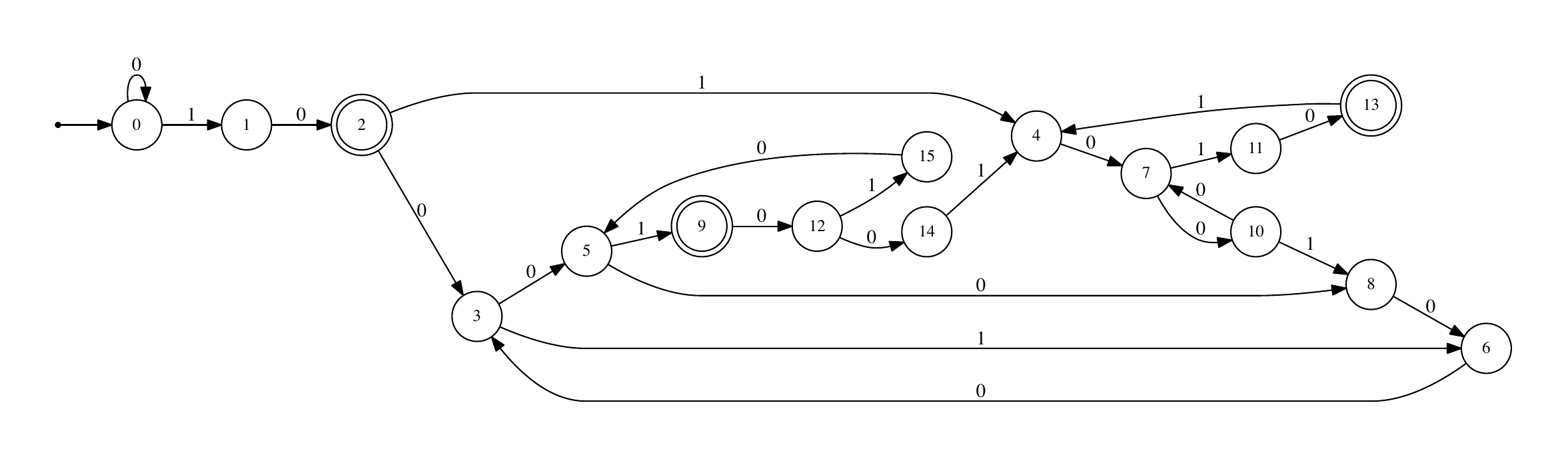}
\end{center}
\caption{Automaton accepting Zeckendorf representation of those $n$ whose $\varphi$-representation has exactly one odd exponent.}
\label{fig8}
\end{figure}
 \end{theorem}

\begin{proof}
The proof of the first two claims mirrors the proof of Theorem~\ref{them7}.  Here is the {\tt Walnut} code:
\begin{verbatim}
def one_odd_exponent "?msd_fib ($oneeven1(x) & $noodd1(y)) |
   ($noeven1(x) & $oneodd1(y))":
def thm10 "?msd_fib Ex,y $saka(n,x,y) & $one_odd_exponent(x,y)":
\end{verbatim}
This creates the automaton in Figure~\ref{fig8}.

The second claim can be handled similarly, but there is one minor difficulty, since the theorem specifies we can only use even indexes $2i$ for $i \geq 2$.  In other words, the last bit and third-to-last bit of the Lucas representation must be $0$.  We handle this with the following {\tt Walnut} code, reusing some automata from earlier in the paper.
\begin{verbatim}
def sum_even_lucas "?msd_fib Ey $noeven1(y) & $bit1(y,0) & $bit3(y,0) & 
   $luctofib(y,n)":
eval thm10b "?msd_fib An $sum_even_lucas(n-2) <=> $thm10(n)":
\end{verbatim}
And {\tt Walnut} returns {\tt TRUE}.

For the last claim, we use the following {\tt Walnut} code:
\begin{verbatim}
reg fibmatch msd_fib msd_fib "([0,0]|[1,0])*[1,1]([0,0]|[1,0])*":
# accepts x and y if the Fibonacci number y 
# occurs in the Zeckendorf expansion of x
eval thm10c "?msd_fib An $thm10(n) => Ex,y $saka(n,x,y) & $fibmatch(x,2)":  
\end{verbatim}
It produces {\tt TRUE} once again.
\end{proof}

Finally, we can also examine those integers having exactly two odd exponents in their base-$\varphi$ expansion.

\begin{theorem}
If $n$ has exactly two odd exponents in its base-$\varphi$ expansions, then these exponents are either
$3$ and $1$, or
$2i+1$ and $1-2i$ for some $i \geq 1$.  Furthermore, each possibility occurs.
\end{theorem}

\begin{proof}
We use {\tt Walnut}.
\begin{verbatim}
reg twoodd1 {0,1} "(()|0|1)(0(0|1))*1(0|1)(0(0|1))*1((0|1)0)*":
# has exactly two 1's in an odd position from the end
reg twoeven1 {0,1} "(()|0|1)(0(0|1))*1(0|1)(0(0|1))*1((0|1)0)*(0|1)":
# has exactly two 1's in an even position from the end

def two_odd_exponents "?msd_fib ($oneeven1(x) & $oneodd1(y)) | 
   ($twoeven1(x) & $noodd1(y)) | ($noeven1(x) & $twoodd1(y))":
def thm11 "?msd_fib Ex,y $saka(n,x,y) & $two_odd_exponents(x,y)":

reg isevenfib msd_fib "0*1(00)*":

def fib3 "?msd_fib Eb,c $isevenfib(a) & $shiftl(a,b) & 
   $shiftl(b,c) & $shiftl(c,d)":
# does (a,d) = (F_{2i}, F_{2i+3})?

def thm11_match "?msd_fib $fib3(a,d) & $fibmatch(x,d) & $fibmatch(y,a)":
# accepts (a,d,x,y) if a = F_{2i}, d = F_{2i+3} and
# F_{2i+3} shows up in x and F_{2i} shows up in y
   
eval thm11b "?msd_fib An $thm11(n) =>  Ex,y $saka(n,x,y) & (($fibmatch(x,5) 
   & $fibmatch(x,2)) | Ea,d $thm11_match(a,d,x,y))":
# if there are two odd exponents for n
# then they are (3,1) or (2i+1,1-2i) for some i
\end{verbatim}
And {\tt Walnut} returns TRUE.

To complete the proof we need to show that each of the possibilities of the two odd exponents $(3,1)$ and
$(2i+1,1-2i)$ for $i \geq 1$ occurs.
For $(3,1)$ we just observe that
$6 = \varphi^3 + \varphi^1 + \varphi^{-4}$.  For the other exponents we use the following {\tt Walnut} code:
\begin{verbatim}
eval thm11c "?msd_fib Aa,d $fib3(a,d) => En,x,y $saka(n,x,y) & 
   $two_odd_exponents(x,y) & $fibmatch(x,d) & $fibmatch(y,a)":
# for each (2i+1,1-2i) there is an n having exactly two odd exponents
# with these two
\end{verbatim}
And {\tt Walnut} returns {\tt TRUE}. 
\end{proof}

\section*{Acknowledgments}

We thank Michel Dekking for pointing out some flaws in a previous version.

\end{document}